\documentclass[reqno]{amsart}
\usepackage{graphicx} 
\usepackage[utf8]{inputenc}
\usepackage{amsmath}

\usepackage{tikz}

\usetikzlibrary{shapes.geometric}
\usepackage{amsthm}

\usepackage{amsfonts}
\usepackage{caption}
\usepackage{amssymb}
\usepackage{verbatim}
\usepackage{mathtools}
\usepackage{subcaption}
\usepackage{changepage}
\newtheorem{theorem}{Theorem}
\newtheorem{lemma}[theorem]{Lemma}

\newtheorem{definition}[theorem]{Definition}
\numberwithin{equation}{section}

\title{The Maximum Length for Ducci Sequences on $\mathbb{Z}_m^n$ when $n$ is Even}

\author{Mark L. Lewis }
\address{Department of Mathematical Sciences\\
Kent State University\\
Kent, OH 44242}
\email{lewis@math.kent.edu}
\author{Shannon M. Tefft}
\address{Department of Mathematical Sciences\\
Kent State University\\
Kent, OH 44242}
\email{stefft@kent.edu}

\subjclass{20D60, 11B83, 11B50}
\keywords{Ducci sequence, modular arithmetic, length, period, $n$-Number Game}

\date{September 2024}

\begin{document}
\begin{abstract}
Let $D: \mathbb{Z}_m^n \to \mathbb{Z}_m^n$ be defined so 
\[D(x_1, x_2, ..., x_n)=(x_1+x_2 \; \text{mod} \; m, x_2+x_3 \; \text{mod} \; m, ..., x_n+x_1 \; \text{mod} \; m).\]
$D$ is known as the Ducci function and for $\mathbf{u} \in \mathbb{Z}_m^n$, $\{D^{\alpha}(\mathbf{u})\}_{\alpha=0}^{\infty}$ is the Ducci sequence of $\mathbf{u}$. Every Ducci sequence enters a cycle because $\mathbb{Z}_m^n$ is finite. In this paper, we aim to establish an upper bound for how long it will take for a Ducci sequence in $\mathbb{Z}_m^n$ to enter its cycle when $n$ is even.
\end{abstract}
\maketitle

\section{Introduction}
\indent We focus on an endomorphism, $D$, on $\mathbb{Z}_m^n$ such that 
\[D(x_1, x_2, ..., x_n)=(x_1+x_2 \; \text{mod} \; m, x_2+x_3 \; \text{mod} \; m, ..., x_n+x_1 \; \text{mod} \; m).\]
We call $D$ the \textbf{Ducci function}, similar to \cite{Breuer1, Ehrlich, Glaser}. If $\mathbf{u} \in \mathbb{Z}_m^n$, then $\{D^{\alpha}(\mathbf{u})\}_{\alpha=0}^{\infty}$ is known as the \textbf{Ducci sequence of} $\mathbf{u}$. 

\indent Because $\mathbb{Z}_m^n$ is finite,  every Ducci sequence  eventually enters a cycle. We have a specific name for this cycle, which we give in the following definition:

\begin{definition}
The \textbf{Ducci cycle of} $\mathbf{u}$ is  
\[\{\mathbf{v} \mid \exists \alpha \in \mathbb{Z}^+ \cup \{0\}, \beta \in \mathbb{Z}^+  \ni \mathbf{v}=D^{\alpha+\beta}(\mathbf{u})=D^{\alpha}(\mathbf{u})\}\].
The \textbf{length of} $\mathbf{u}$, $\mathbf{Len(u)}$, is the smallest $\alpha$ satisfying the equation 
\[\mathbf{v}=D^{\alpha+\beta}(\mathbf{u})=D^{\alpha}(\mathbf{u})\]
 for some $v \in \mathbb{Z}_m^n$ \and the \textbf{period of} $\mathbf{u}$, $\mathbf{Per(u)}$, is the smallest $\beta$ that satisfies the equation. 
\end{definition}

\indent To see this in action, let us look at the Ducci sequence of $(0,0,0,1) \in \mathbb{Z}_5^4$: $(0,0,0,1), (0,0,1,1),(0,1,2,1), (1,3,3,1), (4,1,4,2), (0,0,1,1)$. From here, we can see that the Ducci cycle of $(0,0,0,1)$ is $(0,0,1,1),(0,1,2,1), (1,3,3,1), (4,1,4,2),$ $(0,0,1,1)$. We can also determine that $\text{Len}(0,0,0,1)=1$ and $\text{Per}(0,0,0,1)=4$. 

\indent This tuple $(0,0,0,1)$ and any tuple of the form $(0,0,...,0,1) \in \mathbb{Z}_m^n$ is important when it comes to Ducci sequences. The Ducci sequence of $(0,0,...,0,1) \in \mathbb{Z}_m^n$ is called the \textbf{basic Ducci sequence of } $\mathbf{\mathbb{Z}_m^n}$. This definition is first used on page 302 by \cite{Ehrlich} and also by \cite{Breuer1, Dular, Glaser}. We also define $\mathbf{P_m(n)}=\text{Per}(0,0,...,0,1)$ and $\mathbf{L_m(n)}=\text{Len}(0,0,...,0,1)$. Both of these notations are very similar to how Definition 5 of \cite{Breuer1} define them, and for $Per(0,0,...,0,1)$, the notation is also like \cite{Ehrlich, Glaser}. Therefore, our example from the previous paragraph tells us $P_5(4)=4$ and $L_5(4)=1$. These values are significant because by Lemma 1 in \cite{Breuer1}, if $\mathbf{u} \in \mathbb{Z}_m^n$, then $\text{Len}(\mathbf{u}) \leq L_m(n)$ and $\text{Per}(\mathbf{u})|P_m(n)$. Therefore, $P_m(n)$ and $L_m(n)$ provide a maximal value for $Per(\mathbf{u})$ and $Len(\mathbf{u})$ respectively for any $\mathbf{u} \in \mathbb{Z}_m^n$. The notation of $P_m(n)$ representing the maximum value for a tuple in $\mathbb{Z}_m^n$ is also used on page 858 of \cite{Breuer2}. 

\indent For the rest of this paper, we are most interested in the value of $L_m(n)$, particularly when $n$ is even.  Our goal is to prove:
\begin{theorem}\label{MainTheorem}
    Let $n$ be even. Then
    \begin{enumerate}
        \item If $gcf(n,m)=1$, then $L_m(n)=1$.\\
        \item If there exists $p$ prime, $k, n_1, m_1 \in \mathbb{Z}^+$,  such that $n=p^kn_1$, $m=pm_1$, $gcf(n_1, m_1)=1$ and $p \nmid n_1, m_1$, then $L_m(n)=p^k$.\\
        \item If there exists $p$ prime and $k,l, n_1, m_1 \in \mathbb{Z}^+$ such that $n=p^kn_1$, $m=p^lm_1$, $gcf(n_1, m_1)=1$, and $p \nmid n_1, m_1$ then $L_m(n) \leq p^{k-1}(l(p-1)+1)$.\\
        \item If there exists $p_1, p_2, ..., p_t$ prime where $p_1<p_2< \cdots < p_t$, for  $1 \leq i \leq t$, and $n=p_1^{k_1}p_2^{k_2} \cdots p_t^{k_t}n_1$, $m=p_1^{l_1}p_2^{l_2} \cdots p_t^{l_t}m_1$ with $gcf(n_1, m_1)=1$ and $p_i \nmid n_1, m_1$ for every $i$, then 
        \[L_m(n)=max\{L_{p_i^{l_i}}(n) \; | \; 1 \leq i \leq t \}.\]
    \end{enumerate}
\end{theorem}

\indent It is worth noting that in part (3) of Theorem \ref{MainTheorem}, we believe 
\[L_m(n) = p^{k-1}(l(p-1)+1).\]
 We are able to confirm this is true for all  $m \leq 50$ and even $n \leq 20$ where $gcf(n,m)$ is a power of a prime. We would test for larger $n,m$, but our program for computing $L_m(n)$ requires first finding $P_m(n)$, which typically gets larger as $n,m$ increase. The values of $P_m(n)$ get too large for our Matlab program to find. 

\indent The work in this paper was done while the second author was a Ph.D. student at Kent State University under the advisement of the first author and will appear as part of the second author's dissertation. 

\section{Background}
\indent The Ducci function is originally defined as an endomorphism on $\mathbb{Z}^n$ or $(\mathbb{Z}^+ \cup \{0\})^n$ such that $\bar{D}(x_1, x_2, ..., x_n)=(|x_1-x_2|, x_2-x_3|, ...., |x_n-x_1|)$, and is the most common definition of the Ducci function. Some examples of papers that define the Ducci function this way are \cite{Breuer1, Ehrlich, Freedman, Glaser, Furno}. Note that if $\bar{D}$ is defined on $\mathbb{Z}^n$, then $\bar{D}(\mathbf{u}) \in (\mathbb{Z}^+ \cup \{0\})^n$. Therefore, for simplicity, we will refer to both of these cases as Ducci on $\mathbb{Z}^n$. Other papers, including \cite{Brown, Chamberland, Schinzel}, use the same formula for $\bar{D}$ but define Ducci on $\mathbb{R}^n$. It is of course necessary that we handle the cases of Ducci on $\mathbb{Z}^n$ and on $\mathbb{R}^n$ separately.

\indent Ducci sequences in both of these cases also always enter a cycle. There are discussions of why this happens on $\mathbb{Z}^n$ in \cite{Burmester, Ehrlich, Glaser, Furno} and Theorem 2 of \cite{Schinzel} proves it for Ducci on $\mathbb{R}^n$. A well known fact for Ducci on $\mathbb{Z}^n$ is proved in Lemma 3 of \cite{Furno}: all of the entries of a tuple in a Ducci cycle belong to $\{0,c\}$ for some $c \in \mathbb{Z}^+$. Since $\bar{D}(\lambda \mathbf{u})=\lambda \bar{D}(\mathbf{u})$ for all $\mathbf{u} \in \mathbb{Z}^n$, this means that we can focus on when Ducci is defined on $\mathbb{Z}_2$ where we are using our definition of $D$ given at the beginning of the paper, particularly when examining Ducci cycles. Theorem 1 of \cite{Schinzel} proves a similar finding for Ducci on $\mathbb{R}^n$, namely that if the Ducci sequence reaches a limit point, then all of the entries of that limit point belong to $\{0,c\}$, where this time we have $c \in \mathbb{R}$.

\indent Because of the significance of Ducci on $\mathbb{Z}_2^n$ to the original Ducci case, this leads us to wonder what happens if we define Ducci on $\mathbb{Z}_m^n$ using our definition given at the beginning of the paper. The first paper to look at Ducci on $\mathbb{Z}_m^n$ is \cite{Wong}, and is also examined in \cite{Breuer1, Breuer2, Dular}.

\indent There are a few classical results from the case of Ducci on $\mathbb{Z}^n$ that apply to the case when $m=2$.
When looking at $L_2(n)$ in particular, \cite{Ehrlich} is the first to give that $L_2(n)=1$ when $n$ is odd on page 303. When $n$ is even, Theorem 6 of \cite{Glaser} gives that if $n=2^{k_1}+2^{k_2}$ where $k_1>k_2 \geq 0$, then $L_2(n)=2^{k_2}$. This is then extended by Theorem 4 of \cite{Breuernote} to all even $n=2^kn_1$ where $n_1$ is odd. Here, $L_2(n)=2^k$. Notice that this supports part 2 of Theorem \ref{MainTheorem} when $m=p=2$.

\indent If we allow $n$ to be odd, we have a formula for $L_m(n)$ from Theorem 2 of \cite{Paper3}. Specifically, if $m=2^lm_1$ where $n,m_1$ are odd, then $L_m(n)=l$.
For a case where $n$ is even, Theorem 2 of \cite{Paper1}, proves that if $n=2^k$ and $m=2^l$, then $L_m(n)=2^{k-1}(l+1)$. 

\indent Before moving on, we would also like to give a few definitions that will be useful later. If $\mathbf{u}, \mathbf{v} \in \mathbb{Z}_m^n$, then $\mathbf{v}$ is a \textbf{predecessor} to $\mathbf{u}$ if $D(\mathbf{v})=\mathbf{u}$. 

\indent We let $\mathbf{K(\mathbb{Z}_m^n)}$ be the set of all tuples in a Ducci cycle. On page 6001 of \cite{Breuer1}, it is stated that $K(\mathbb{Z}_m^n)$ is a subgroup of $\mathbb{Z}_m^n$. A proof of this is provided in Theorem 1 of \cite{Paper1}.

\indent We now consider another example of a Ducci sequence and its cycle. For this, we will look at the basic Ducci sequence of $\mathbb{Z}_2^6$. We create a transition graph that not only gives the basic Ducci sequence, but also gives all tuples whose sequence has tuples in the basic Ducci sequence. This transition graph is given in Figure \ref{Example}.

\begin{figure}
\centering
\resizebox{.95\textwidth}{!}{
\begin{tikzpicture}[node distance={30mm}, thick, main/.style = {draw, circle}]

\node[main](1){$(0,0,0,1,0,1)$};
\node[main](2)[right of=1]{$(0,0,1,1,1,1)$};
\node[main](3)[below right of =2]{$(0,1,0,0,0,1)$};
\node[main](4)[below left of=3]{$(1,1,0,0,1,1)$};
\node[main](5)[left of=4]{$0,1,0,1,0,0)$};
\node[main](6)[above left of=5]{$(1,1,1,1,0,0)$};

\node[main](7)[above left of=1]{$(0,0,0,0,1,1)$};
\node[main](8)[above of=7]{$(0,0,0,0,0,1)$};
\node[main](9)[left of=7]{$(1,1,1,1,1,0)$};

\node[main](10)[above right of=2]{$(1,1,1,0,1,1)$};
\node[main](11)[above of=10]{$(0,1,0,1,1,0,1)$};
\node[main](12)[right of=10]{$(1,0,1,0,0,1,0)$};

\node[main](13)[right of=3]{$(1,1,0,0,0,0)$};
\node[main](14)[above right of=13]{$(0,1,0,0,0,0)$};
\node[main](15)[below right of=13]{$(1,0,1,1,1,1)$};

\node[main](16)[below right of=4]{$(1,0,1,1,1,0)$};
\node[main](17)[right of=16]{$(0,1,1,0,1,0)$};
\node[main](18)[below of =16]{$(1,0,0,1,0,1)$};

\node[main](19)[below left of=5]{$(0,0,1,1,0,0)$};
\node[main](20)[left of=19]{$(0,0,0,1,0,0)$};
\node[main](21)[below of=19]{$(1,1,1,0,1,1)$};

\node[main](22)[left of=6]{$(1,0,1,0,1,1)$};
\node[main](23)[above left of=22]{$(0,1,1,0,0,1)$};
\node[main](24)[below left of=22]{$(1,0,0,1,1,0)$};

\draw[->](1)--(2);
\draw[->](2)--(3);
\draw[->](3)--(4);
\draw[->](4)--(5);
\draw[->](5)--(6);
\draw[->](6)--(1);

\draw[->](7)--(1);
\draw[->](8)--(7);
\draw[->](9)--(7);

\draw[->](10)--(2);
\draw[->](11)--(10);
\draw[->](12)--(10);

\draw[->](13)--(3);
\draw[->](14)--(13);
\draw[->](15)--(13);

\draw[->](16)--(4);
\draw[->](17)--(16);
\draw[->](18)--(16);

\draw[->](19)--(5);
\draw[->](20)--(19);
\draw[->](21)--(19);

\draw[->](22)--(6);
\draw[->](23)--(22);
\draw[->](24)--(22);

\end{tikzpicture}
}
\caption{Transition Graph for $\mathbb{Z}_2^6$}\label{Example}
\end{figure}
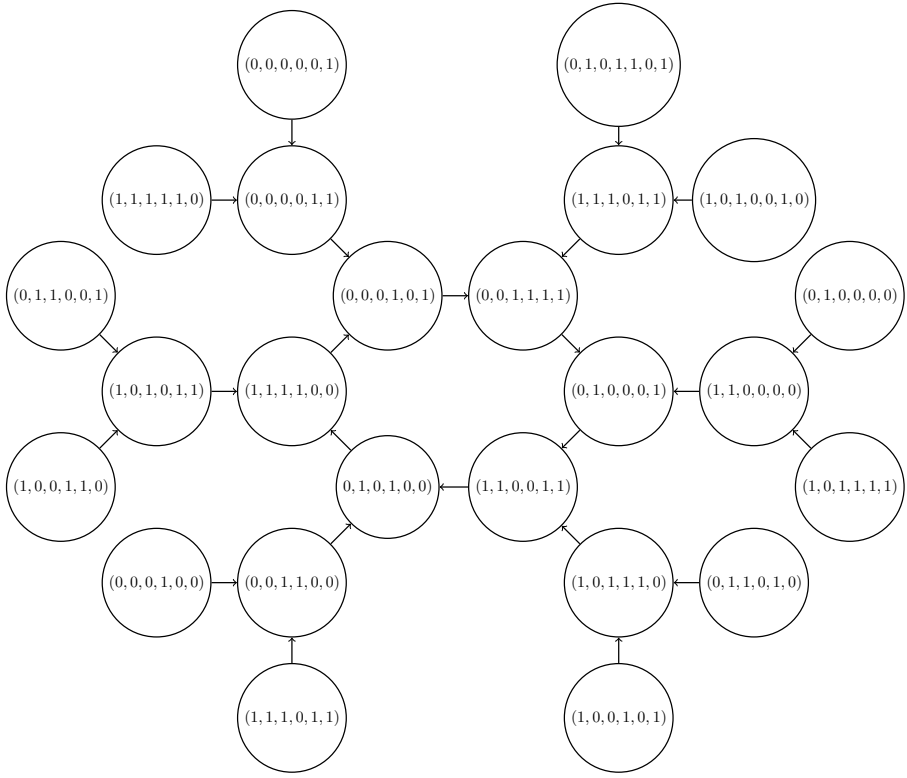


\indent Notice that we can determine that $L_2(6)=2$ from Figure \ref{Example}, which agrees with Theorem \ref{MainTheorem}. We can also see that $P_2(6)=6$. 

\indent It is worth noting that every tuple in Figure \ref{Example} that has a predecessor has exactly two, which is the same as our $m$ in this case. This is because of Theorem 4 in \cite{Paper1}, which says that if $n$ is even, then every tuple that has a predecessor has exactly $m$ predecessors. This theorem also tells us that if $\mathbf{u} \in \mathbb{Z}_m^n$ has a predecessor, call it $(x_1, x_2, ..., x_n)$, then the remaining predecessors are of the form 
\[(x_1+z, x_2-z, x_3+z, ..., x_n-z)\]
 for some $z \in \mathbb{Z}_m$ and all tuples of this form are a predecessor to $\mathbf{u}$. 
 
 \indent All of the tuples in the transition graph in Figure \ref{Example} that have a predecessor, call them $(x_1, x_2, ..., x_n)$, also satisfy the condition that $x_1-x_2+x_3- \cdots -x_n \equiv 0 \; \text{mod} \; 2$. In fact,we have the following theorem about this:
 
 \begin{theorem} \label{n_even_m_odd_preds}
Let $n$ be even. $(x_1, x_2,..., x_n) \in \mathbb{Z}_m^n$ has a predecessor, if and only if $x_1-x_2+x_3-x_4+ \cdots+x_{n-1}-x_n \equiv 0 \; \text{mod} \; m$
\end{theorem}
We first note that \cite{Furno} proves in Lemma 5 that this is true for $m=2$ and \cite{Breuer1} proves it is true when $m$ is prime in Lemma 4.
\begin{proof}[Proof of Theorem \ref{n_even_m_odd_preds}]
$(\Rightarrow)$ Assume $(x_1, x_2, ..., x_n)$ has a predecessor $(y_1, y_2, ..., y_n)$. Then 
\[y_1+y_2\equiv x_1 \; \text{mod} \; m\]
\[y_2+y_3 \equiv x_2 \; \text{mod} \; m\]
\[\vdots\]
\begin{equation}\label{Eq_pred}
y_n+y_1 \equiv x_n \; \text{mod} \; m.
\end{equation}
Subtracting the second equation from the first, we get
\[y_1-y_3 \equiv x_1-x_2 \; \text{mod} \; m.\]
Adding this to the third equation, we get
\[y_1+y_4 \equiv x_1-x_2+x_3 \; \text{mod} \; m.\]
Continuing, we get 
\[y_1-y_{n-1} \equiv x_1-x_2+ \cdots -x_{n-2} \; \text{mod} \; m\]
and
\[y_1+y_n \equiv x_1-x_2+ \cdots +x_{n-1} \; \text{mod} \; m.\]
Using this and Equation (\ref{Eq_pred}), we see
\[x_n \equiv x_1-x_2+ \cdots +x_{n-1} \; \text{mod} \; m,\]
which gives
\[x_1-x_2+\cdots +x_{n-1}-x_n \equiv 0 \; \text{mod} \; m,\]
and $(\Rightarrow)$ follows.

$(\Leftarrow)$ Assume $x_1-x_2+x_3-x_4+ \cdots+x_{n-1}-x_n \equiv 0 \; \text{mod} \; m$. It suffices to show that there exists $y_1, y_2, ..., y_n$ that satisfy the equations
\[y_1+y_2\equiv x_1 \; \text{mod} \; m\]
\[y_2+y_3 \equiv x_2 \; \text{mod} \; m\]
\[\vdots\]
\[y_n+y_1 \equiv x_n \; \text{mod} \; m.\]
Let $y_1=0$. If we let $y_2=x_1$, then the first equation will be satisfied. If we then let $y_3=x_2-x_1$, the second equation will be satisfied. If we continue this following the structure $y_j=x_{j-1}-x_{j-2}+ \cdots +x_1$ when $j$ is even and $y_j=x_{j-1}-x_{j-2}+ \cdots -x_1$ when $j$ is odd, then the first $n-1$ equations will be satisfied. 

\indent This would result in $y_n=x_{n-1}-x_{n-2}+ \cdots +x_1$. We then see that 
\[y_1+y_n \equiv x_{n-1}-x_{n-2}+ \cdots +x_1 \; \text{mod} \; m\]
By assumption, $x_{n-1}-x_{n-2}+ \cdots +x_1 \equiv x_n\; \text{mod} \; m$ and $y_1+y_n \equiv x_n \; \text{mod} \; m$ follows, and the final equation is satisfied. Therefore, $(x_1, x_2, ..., x_n)$ has at least one predecessor, $(y_1, y_2, ..., y_n)$. 
\end{proof}

\indent Moving forward, if $\mathbf{u} \in \mathbb{Z}_m^n$, it will be useful to have a tool that gives us information about tuples in the Ducci sequence of $\mathbf{u}$. To do this, let us first look at the first few tuples in the Ducci sequence of a tuple $(x_1, x_2, ..., x_n) \in \mathbb{Z}_m^n$:

\[(x_1, x_2, ..., x_n)\]
     \[(x_1+x_2, x_2+x_3, ..., x_n+x_1)\]
     \[(x_1+2x_2+x_3, x_2+2x_3+x_4, ..., x_n+2x_1+x_2\]
     \[(x_1+3x_2+3x_3+x_4, x_2+3x_3+3x_4+x_5, ..., x_n+3x_1+3x_2+x_3)\]
     \[\vdots\]   
     
   We can see that the coefficients on the $x_j$ in the first entry of each tuple in the Ducci sequence occurs in all of the entries of that tuple for some $x_i$. For this reason, we will let $a_{r,s}$ represent the coefficient that is on $x_s$ in the first entry of $D^r(x_1, x_2, ..., x_n)$. Here $r \geq 0$ and $1 \leq s \leq n$. The $s$ coordinate reduces modulo $n$, with the exception that we write $a_{r,n}$ instead of $a_{r,0}$. Note that $a_{r,s}$ also appears as the coefficient on $x_{s-i+1}$ in the $i$th entry of $D^r(x_1, x_2, ..., x_n)$. Page 6 of \cite{Paper1} provides a more thorough explanation of why we can do this. We also have that 
   \[D^r(0,0,...,0,1)=(a_{r,n}, a_{r, n-1}, ..., a_{r,1}).\]
   
   Additionally, by Theorem 5 of \cite{Paper1}, when $r<n$, $a_{r,s}=\displaystyle{\binom{r}{s-1}}$. This theorem also tells us that $a_{r+t,s}=\displaystyle{\sum_{i=1}^n a_{t,i}a_{r,s-i+1}}$ when $t \geq 1$ and $a_{r,s}=a_{r-1,s}+a_{r-1,s-1}$.

\section{Proving the Main Theorem}
Before we can prove Theorem \ref{MainTheorem}, there are a few lemmas that we will need. We begin with two lemmas about particular binomial coefficients. We believe that both Lemmas \ref{binom_coeff_p^a-1} and \ref{binom_power_of_p_diff} are known and have been proven before, but we are providing proofs here for completeness.

\begin{lemma}\label{binom_coeff_p^a-1}
    Let $p$ be prime and $j \leq p^k-1$, $k \geq 1$. Then 
    \[\binom{p^k-1}{j} \equiv (-1)^j \; \text{mod} \; p.\]
\end{lemma}
\begin{proof}
To prove this, we will do induction on $j$.

\textbf{Basis Step} $\mathbf{j=0}$:  Note that  this is satisfied by $\displaystyle{\binom{p^k-1}{0}}=1$.

 \textbf{Inductive Step:} Now assume that $\displaystyle{\binom{p^k-1}{j-1}} \equiv (-1)^{j-1}  \; \text{mod} \; p$. Note 
    \[\binom{p^k}{j} =\binom{p^k-1}{j}+\binom{p^k-1}{j-1}.\]
    Because we proved the $j=0$ case in the basis step, we can assume $0< j \leq p^k-1$, which will give us 
    \[\binom{p^k-1}{j}+\binom{p^k-1}{j-1} \equiv 0 \; \text{mod} \; p.\]
    Solving for $\displaystyle{\binom{p^k-1}{j}}$, we get
    \[\binom{p^k-1}{j} \equiv -\binom{p^k-1}{j-1} \; \text{mod} \; p.\]
    By induction, this is equivalent to
    \[ -(-1)^{j-1} \; \text{mod} \; p\]
    or
    \[ (-1)^j \; \text{mod} \; p,\]
    and the lemma follows. 

\end{proof}

The next lemma is similar to Lemma \ref{binom_coeff_p^a-1}:

\begin{lemma}\label{binom_power_of_p_diff}
For $k > 1$ and $p$ an odd prime
    \[\binom{p^k-p^{k-1}}{j}\equiv 
    \begin{cases}
    
    (-1)^c \; \text{mod} \; p & j=cp^{k-1}\\
        0\; \text{mod} \; p & otherwise\\
    \end{cases}.\]
\end{lemma}
\begin{proof}
   We prove this via induction on $j$.
   
    \textbf{Basis steps} $\mathbf{j=0,1}$: These follow from
   \[\binom{p^k-p^{k-1}}{0} \equiv 1 \; \text{mod} \; p\]
   and
   \[\binom{p^k-p^{k-1}}{1} \equiv 0 \; \text{mod} \; p.\]
   
   \textbf{Inductive step:} Assume that the lemma is true for $j^*<j$. For $j>0$, we use the Chu-Vandermonde identity, a proof of which can be found in \cite{Spivey} to yield
   \begin{equation}\label{Eq_p_binom}
   \binom{p^k}{j}=\sum_{i=0}^j \binom{p^{k-1}}{i} \binom{p^k-p^{k-1}}{j-i}\equiv 0 \; \text{mod} \; p.
   \end{equation}
   If $j < p^{k-1}$, the sum in (\ref{Eq_p_binom}) is also equivalent to
   \[\binom{p^{k-1}}{0}\binom{p^k-p^{k-1}}{j} \; \text{mod} \; p,\]
   which is congruent to $0 \; \text{mod} \; p$, so $\displaystyle{\binom{p^k-p^{k-1}}{j}} \equiv 0 \; \text{mod} \; p$.
   
\indent If $j \geq p^{k-1}$, the sum in (\ref{Eq_p_binom}) is equivalent to
   \[\binom{p^{k-1}}{0}\binom{p^k-p^{k-1}}{j}+\binom{p^{k-1}}{p^{k-1}}\binom{p^k-p^{k-1}}{j-p^{k-1}}.\]
   This is congruent to $0 \; \text{mod} \; p$ by Equation (\ref{Eq_p_binom}), we have
   \[\binom{p^k-p^{k-1}}{j} \equiv -\binom{p^k-p^{k-1}}{j-p^{k-1}} \; \text{mod} \; p.\]
   If $j \neq cp^{k-1}$ for some $c$, then $\displaystyle{\binom{p^k-p^{k-1}}{j}}$ is congruent to $0 \; \text{mod} \; p$ by induction. If $j=cp^{k-1}$ for some $c$, then $\displaystyle{\binom{p^k-p^{k-1}}{j}}$ is equivalent to $-(-1)^{c-1} \; \text{mod} \; p$ or $(-1)^c \; \text{mod} \; p$ and the lemma follows.
    
\end{proof}

\indent Next, we would like to look at the $a_{r,s}$ coefficient for a particular $r$ in  Lemmas \ref{p_dividing_multiplecoeff} and \ref{coefficient_to_smaller_n_broader}.

\begin{lemma}\label{p_dividing_multiplecoeff}
    Let $n=p^kn_1$ where $p$ is prime, $n_1>1$ and $c \geq 1$. Then for $s \neq bp^k+1$ and $0 \leq b < n_1$,
    $a_{cp^k,s} \equiv 0 \; \text{mod} \; p$
\end{lemma}
\begin{proof}
    We prove this via induction on $c$. 
    
    \textbf{Basis Step} $\mathbf{c=1}$: 
    \[a_{p^k,s} =\binom{p^k}{s-1} \equiv 0 \; \text{mod} \; p\]
    precisely when $s \neq 1, p^k+1$.
    
    \textbf{Inductive Step:} Assume 
    \[a_{(c-1)p^k,s} \equiv 0 \; \text{mod} \; p\]
     for $s \not \equiv b'p^k+1 \; \text{mod} \; n, \; 0 \leq b' \leq n_1-1$. Using Theorem 5 of \cite{Paper1}
    \[a_{cp^k,s}= \sum_{i=1}^n a_{p^k,i}a_{(c-1)p^k,s-i+1}.\]
    Using $a_{p^k,i} =\displaystyle{\binom{p^k}{i-1}}$, this sum is equivalent to
    \[a_{p^k,1}a_{(c-1)p^k,s}+a_{p^k,p^k+1}a_{(c-1)p^k,s-p^k} \; \text{mod} \; p.\]
    Notice that for $s \not \equiv bp^k+1 \; \text{mod} \; n$, this is $\equiv 0 \; \text{mod} \; p$ where $0 \leq b \leq n_1-1$ by induction, and the lemma follows. 
\end{proof}

For the remaining cases where $s=bp^k+1$, we have this lemma:
\begin{lemma} \label{coefficient_to_smaller_n_broader}
    Let $n=p^kn_1$ for some $k$ and $n_1>1$, $a_{r,s}$ be the coefficients for $n$, and $a^*_{r,s}$ be the coefficients for $n_1$.
     Then, for $c \geq 1$ and $0 \leq b <n_1$,
     \[a_{cp^k,bp^k+1} \equiv a^*_{c,b+1} \; \text{mod} \; p.\]
\end{lemma}
\begin{proof}
We prove this via induction on c:

 \textbf{Basis Step} $\mathbf{c=1}$:
Notice that because $p^k<n$, we have
$a_{p^k,1}=\displaystyle{\binom{p^k}{0}}$, which is $1$. Since $a^*_{1,1}=\displaystyle{\binom{1}{0}=1}$, we have that 
\[a_{p^k,1} \equiv a^*_{1,1} \; \text{mod} \; p.\]
Similarly, 
\[a_{p^k,p^k+1}=\binom{p^k}{p^k}=1\]
and 
\[a^*_{1,2}=\binom{1}{1}=1,\]
so $a_{p^k,p^k+1} \equiv a^*_{1,2} \; \text{mod} \; p$. Finally, for $2 \leq b <n_1$, because $\displaystyle{\binom{p^k}{j_1} \equiv 0 \; \text{mod} \; p}$ when $j_1 \neq 1, p^k$ and $\displaystyle{\binom{1}{j_2}=0}$ when $j_2>1$, we have
\[a_{p^k,bp^k+1} \equiv a^*_{1,b+1} \; \text{mod} \; p\]
for every $2 \leq b <n_1$. The basis case follows.

 \textbf{Inductive Step:} Assume
\[a_{(c-1)p^k,bp^k+1} \equiv a^*_{c-1,b+1} \; \text{mod} \; p\]
when $0 \leq b <n_1$. Then, 
\begin{equation}\label{Eq_in_smaller_n_coeff}
a_{cp^k, bp^k+1}=\sum_{i=1}^n a_{p^k,i}a_{(c-1)p^k, bp^k+2-i}.
\end{equation}
Notice that $a_{p^k,s}=\displaystyle{\binom{p^k}{s-1}} \equiv 0 \; \text{mod} \; p$ when $s \neq 1, p^k+1$.
 Therefore, our remaining pieces tell us Equation (\ref{Eq_in_smaller_n_coeff}) is 
\[a_{p^k,1}a_{(c-1)p^k,bp^k+1}+a_{p^k,p^k+1}a_{(c-1)p^k,(b-1)p^k+1} \; \text{mod} \; p.\]
By induction, this is 
\[a^*_{c-1,b+1}+a^*_{c-1,b} \; \text{mod} \; p,\]
which is $a^*_{c,b+1}$. The lemma follows.

\end{proof}

\indent Putting Lemmas \ref{p_dividing_multiplecoeff} and \ref{coefficient_to_smaller_n_broader} together allows us to visualize what $D^{cp^k}(0,0,...,0,1)$ looks like in $\mathbb{Z}_p^n$ for $n=p^kn_1$: 
\[D^{cp^k}(0,0,...,0,1)=(0,0,...,0,a^*_{c,n_1}, 0, ...,0, a^*_{c, n_1-2}, 0, ..., 0, a^*_{c,1}).\]
\indent Bear in mind that the last set of $...$ covers a much larger area then the first two.

\indent We now have the pieces we need to prove Theorem \ref{MainTheorem}:

\begin{proof}[Proof of Theorem \ref{MainTheorem}]
\textbf{(1)}: We would like to first note that when $m$ is prime, this case follows from Lemma 5 of \cite{Breuer1}.

Assume that $gcf(n,m)=1$. We want to prove 
 \[(0,0,...,0,1,1) \in K(\mathbb{Z}_m^n)\]
 and 
 \[(0,0,...,0,1) \not \in K(\mathbb{Z}_m^n).\]
  We first use Theorem \ref{n_even_m_odd_preds} to note that since $-1 \not \equiv 0 \; \text{mod} \; m$, $(0,0,...,0,1)$ does not have a predecessor and therefore $(0,0,...,0,1) \not \in K(\mathbb{Z}_m^n)$. 

    We now note that $(0,0,...,0,1,1)$ has $m$ predecessors and that if any of these predecessors is in the cycle, then so is $(0,0,...,0,1,1)$. In addition to this, if there exists $\mathbf{u}, \mathbf{v} \in \mathbb{Z}_m^n$ such that $D^2(\mathbf{v})=D(\mathbf{u})=(0,0,...,0,1,1)$, then $\mathbf{u} \in K(\mathbb{Z}_m^n)$ because otherwise we would have 
    \[\text{Len}(\mathbf{v})>\text{Len}(\mathbf{u})=\text{Len}(0,0,...,0,1)=L_m(n),\]
     which cannot happen.
    
   \indent Since $(0,0,...,0,1)$ is a predecessor for $(0,0,...,0,1,1)$, then by Theorem 4 of \cite{Paper1}, the other predecessors of $(0,0,...,0,1,1)$ will be of the form 
   \[\mathbf{u}=(z,m-z,z,...,z,1-z)\]
    for some nonzero $z \in \mathbb{Z}_m$, so we only need to prove that there exists such a $z$ where
   $\mathbf{u}$
    has a predecessor. By Theorem \ref{n_even_m_odd_preds}, $\mathbf{u}$ has a predecssor if and only if 
    \[z-(m-z)+z-(m-z)+ \cdots +z-(1-z) \equiv 0 \; \text{mod} \; m,\]
    or equivalently,
    \[nz-1 \equiv 0 \; \text{mod} \; m.\]
    Since $gcf(n,m)=1$, there exists a unique $z$ such that $nz \equiv 1 \; \text{mod} \; m$. Therefore we have a $z$ that satisfies the above and $(z,m-z,z,...,z,1-z)$ has a predecessor. Therefore, $L_m(n)=1$.
    
    \textbf{(2)}: It is worth noting that if $m=p$, then this case follows from Theorem 4 of \cite{Breuer1}.
    Assume $n=p^kn_1$, $m=pm_1$ where $gcf(n_1, m_1)=1$ and $p \nmid n_1, m_1$.
    We must then start with proving that  $L_m(n)=p^k$.
    
   \indent Let $P_m(n)=d$. From the first case of this theorem, we know $\displaystyle{L_{\frac{m}{p}}(n)}=1$, which means 
    \[a_{d+1,s} \equiv a_{1,s} \; \text{mod} \; \frac{m}{p}\]
    or
    \[ a_{d+1,s} \equiv 
    \begin{cases}
    \begin{aligned}
        1 \; \text{mod} \; \displaystyle{\frac{m}{p}} && s=1,2\\
        0 \; \text{mod} \; \displaystyle{\frac{m}{p}} && \text{otherwise}\\
        \end{aligned}
    \end{cases}.\]
    We will use this to write
    \[a_{d+1,s}\equiv 
    \begin{cases}
    \begin{aligned}
        \delta_s \displaystyle{\frac{m}{p}} +1 \; \text{mod} \; m && s=1,2\\
         \delta_s \displaystyle{\frac{m}{p}}\; \text{mod} \; m && \text{otherwise}\\
         \end{aligned}
    \end{cases}\]
    where $0 \leq \delta_s <p$. We start by showing $a_{p^k+d,s} \equiv a_{p^k,s} \; \text{mod} \; m$ for every $s$. First,
    \[a_{p^k+d,s}=\sum_{i=1}^n a_{d+1,i}a_{p^k-1,s-i+1}.\]
    Substituting in our values for $a_{d+1,s}$, we see this is equivalent to
    \[(\delta_1\frac{m}{p}+1)a_{p^k-1,s}+(\delta_2\frac{m}{p}+1)a_{p^k-1,s-1}+\frac{m}{p}\sum_{i=3}^n \delta_ia_{p^k-1,s-i+1} \; \text{mod} \; m,\]
    which after some rearranging is
    \[a_{p^k-1,s}+a_{p^k-1,s-1}+\frac{m}{p}\sum_{i=1}^n \delta_ia_{p^k-1,s-i+1} \; \text{mod} \; m.\]
    Using $a_{r,s}=a_{r-1,s}+a_{r-1,s-1}$ from Theoremm 5 of \cite{Paper1} and substituting $a_{p^k-1,s-i+1}$, this is
    \[ a_{p^k,s}+\frac{m}{p}\sum_{i=1}^n \delta_i\binom{p^k-1}{s-i} \; \text{mod} \; m.\]
  \indent Notice that this basis case will follow if $\displaystyle{\frac{m}{p}\sum_{i=1}^n \delta_i\binom{p^k-1}{s-i} }\equiv 0 \; \text{mod} \; m$, which will follow if $\displaystyle{\sum_{i=1}^n \delta_i\binom{p^k-1}{s-i}} \equiv 0 \; \text{mod} \; p$. 
    
\indent Note that in this sum, $\displaystyle{\binom{p^k-1}{s-i}} \neq 0$ for only $p^k$ many consecutive terms for every $s$. We would like to use this and the following claim to prove 
\[\sum_{i=1}^n \delta_i\binom{p^k-1}{s-i} \equiv 0 \; \text{mod} \; p.\]
   
    \indent \textbf{Claim:} 
    \begin{enumerate}
        \item $\delta_s=0$ as long as $s \neq bp^k+1, bp^k+2$ for some $0 \leq b < \displaystyle{\frac{n}{p^k}}$.\\
       \item  $\delta_{bp^k+1}=\delta_{bp^k+2},$ when $0 \leq b <\displaystyle{\frac{n}{p^k}}$.\\
       \item  $\delta_{bp^k+1}+\delta_{(b-1)p^k+2} \equiv 0 \; \text{mod} \; p$ where $0 \leq b <\displaystyle{\frac{n}{p^k}}$.\\
    \end{enumerate}
   Notice that  $\displaystyle{\sum_{i=1}^n \delta_i\binom{p^k-1}{s-i}} \equiv 0 \; \text{mod} \; p$ will follow if the claim is true: the claim will give us that most of the $\delta_i$ will be congruent to $0 \; \text{mod} \; p$, and since at most $p^k$ terms in the sum will already be nonzero, there will only be $2$ nonzero terms in the entire sum, as for every $p^k$ consecutive $\delta_i$ you pick, only $2$ will be nonzero. This leaves us with 2 cases:
   
   \indent \textbf{Case 1:} The two nonzero terms in the sum are adjacent to each other. 
   \begin{itemize}
       \item If $p$ is odd, then because of the claim  and Lemma \ref{binom_coeff_p^a-1}, we have $\displaystyle{\sum_{i=1}^n \delta_i\binom{p^k-1}{s-i}}$ is either congruent to
       \[ \delta_{bp^k+1}-\delta_{bp^k+2} \; \text{mod} \; p\]
       or
       \[-\delta_{bp^k+1}+\delta_{bp^k+2} \; \text{mod} \; p,\]
       both of which are equivalent to $0 \; \text{mod}\; p$.\\     
       
       \item If $p=2$, then by the claim and the well-known Lucas's Theorem (a proof of which can be found in Theorem 1 of \cite{Fine}), the sum \[\sum_{i=1}^n \delta_i\binom{2^k-1}{s-i} \equiv \delta_{2^kb+1}+\delta_{2^kb+2} \; \text{mod} \; 2,\]
      which is equivalent to $0 \; \text{mod} \; 2$.
      
   \end{itemize}
   and $L_m(n) \leq p^k$ would follow.
   
   \indent \textbf{Case 2:} The two nonzero terms are $p^k$ terms apart. This will only happen if the remaining terms of $\displaystyle{\sum_{i=1}^n \delta_i\binom{p^k-1}{s-i}}$ are equivalent to
   \[\delta_{bp^k+1}\binom{p^k-1}{0}+\delta_{(b-1)p^k+2}\binom{p^k-1}{p^k-1}  \; \text{mod} \; p, \]
 
   which, by the claim, would be congruent to $0 \; \text{mod} \; p$ and $L_m(n) \leq p^k$ would follow from here as well.
   
   \indent \textbf{Proof of Claim:} Without loss of generality, we may assume that $p^k|d$ and write $d=cp^k$. By Lemma \ref{p_dividing_multiplecoeff}, $a_{d,s} \equiv 0 \; \text{mod} \; p$ where $s \neq bp^k+1$ for some $0 \leq b \leq \displaystyle{\frac{n}{p^k}}$. Then for $s \neq bp^k+1, bp^k+2$,
   \[a_{d+1,s}=a_{d,s}+a_{d,s-1} \equiv 0 \; \text{mod} \; p\]
   which implies $\delta_s=0$ for $s \neq bp^k+1, bp^k+2$ and Part $(1)$ of the Claim follows.
   
   \indent For Part (2) of the Claim, 
   \[a_{d+1,bp^k+1}-a_{d+1,bp^k+2}=a_{d,bp^k+1}+a_{d,bp^k}-a_{d,bp^k+2}-a_{d,bp^k+1},\]
   which is
   \[a_{d,bp^k}-a_{d,bp^k+2} \equiv 0 \; \text{mod} \; p,\]
because of Lemma \ref{p_dividing_multiplecoeff} and because $p^k|d$.
   Therefore, Part $(2)$ of the Claim follows. 
   
   \indent As for Part (3) of the Claim, we have
   \[a_{d+1,bp^k+1}+a_{d+1,(b-1)p^k+2} = a_{d,bp^k+1}+a_{d,bp^k}+a_{d,(b-1)p^k+2}+a_{d,(b-1)p^k+1},\]
   which is equivalent to
   \[a_{d,bp^k+1}+a_{d,(b-1)p^k+1} \; \text{mod} \; p.\]
   By Lemma \ref{coefficient_to_smaller_n_broader}, this is 
   \[ a^*_{c,b+1}+a^*_{c,b} \; \text{mod} \; p\]
   or
   \[ a^*_{c+1,b+1} \;  \text{mod} \; p,\]
   where $a^*_{r,s}$ is the coefficient for $n_1$. By Proposition 4 of \cite{Breuer1}, $P_p(n_1p^k)=p^kP_p(n_1)$. Since 
   \[d=cp^k=p^kP_p(n_1),\] 
   we conclude that $c=P_p(n_1)$.
   So by Part (1) of Theorem \ref{MainTheorem},
\[a_{d+1,bp^k+1}+a_{d+1,(b-1)p^k+2} \equiv 
\begin{cases}
     0 \; \text{mod} \; p & b>1\\
     1 \; \text{mod} \; p & b=0,1\\
\end{cases}.\]
For $b>1$, 
\[a_{d+1,bp^k+1}+a_{d+1,(b-1)p^k+2} \equiv (\delta_{bp^k+1}+\delta_{(b-1)p^k+2})\frac{m}{p} \; \text{mod} \; m,\]
 which will give us $\delta_{bp^k+1}+\delta_{(b-1)p^k+2} \equiv 0 \; \text{mod} \; p$.

For $b=0,1$, 
\[a_{d+1,bp^k+1}+a_{d+1,(b-1)p^k+2}\equiv (\delta_{bp^k+1}+\delta_{(b-1)p^k+2})\frac{m}{p} +1 \; \text{mod} \; m\]
 gives us $\delta_{bp^k+1}+\delta_{(b-1)p^k+2} \equiv 0 \; \text{mod} \; p$ for $b=0,1$ and $(3)$ follows. Since the whole claim follows, this gives us $L_m(n)\leq p^k$.

We now need to prove that $L_m(n)=p^k$. Suppose $L_m(n) \leq p^k-1$. Then $a_{d+p^k-1,s} \equiv a_{p^k-1,s} \; \text{mod} \; m$ for every $s$ and
    \[a_{d+p^k-1,s}=\sum_{i=1}^n a_{d+1,i}a_{p^k-2,s-i+1}.\]
    If we separate out the terms where $s=1,2$, this is
    \[a_{d+1,1}a_{p^k-2,s}+a_{d+1,2}a_{p^k-2,s-1}+\sum_{i=3}^n a_{d+1,i}a_{p^k-2,s-i+1}.\]
    Substituting in our values for $a_{d+1,s}$ and $a_{p^k-2,s-i+1}$ in the sum and rearranging, this is equivalent to
    \[a_{p^k-2,s}+a_{p^k-2,s-1}+\frac{m}{p}\sum_{i=1}^n \delta_s \binom{p^k-2}{s-i} \; \text{mod} \; m\]
    or
   \[ a_{p^k-1,s}+\frac{m}{p}\sum_{i=1}^n \delta_i \binom{p^k-2}{s-i} \; \text{mod} \; m.\]
   This implies that $\displaystyle{\frac{m}{p}\sum_{i=1}^n \delta_i\binom{p^k-2}{s-i}} \equiv 0 \; \text{mod} \; m$, which produces
   \[\sum_{i=1}^n \delta_i\binom{p^k-2}{s-i} \equiv 0 \; \text{mod} \; p.\]
Substituting $\delta_i=0$ for $i \neq bp^k+1, bp^k+2$ from the Claim, we have   
   \[ \sum_{b=0}^{n_1-1} \delta_{bp^k+1}\binom{p^k-2}{s-bp^k-1}+\delta_{bp^k+2}\binom{p^k-2}{s-bp^k-2} \; \text{mod} \; p.\]
   Using our Claim, this is
   \[\sum_{b=0}^{n_1-1} \delta_{bp^k+1}(\binom{p^k-2}{s-bp^k-1}+\binom{p^k-2}{s-bp^k-2}) \; \text{mod} \; p\]
   or
   \[\sum_{b=0}^{n_1-1} \delta_{bp^k+1}\binom{p^k-1}{s-bp^k-1} \; \text{mod} \; p.\]
   Notice that only 1 of the $\displaystyle{\binom{p^k-1}{s-bp^k-1}}$ is not congruent to $0 \; \text{mod} \; p$. Therefore, we get $\displaystyle{\sum_{i=1}^n \delta_i\binom{p^k-2}{s-i}}$ is congruent to either
   \[\delta_{bp^k+1} \; \text{mod} \; p\]
   or
   \[-\delta_{bp^k+1} \; \text{mod} \; p\]
   for some $0 \leq b \leq n_1-1$. Now as long as one of the $\delta_{bp^k+1}$ coefficients is nonzero, we have a contradiction. Suppose then that $\delta_{bp^k+1}=0$ for every $b$. Then, we would have $L_m(n)=1$. Then $(0,0,...,0,1,1)$ has a predecessor that is also in the cycle. This is equivalent, to one of its predecessors having a predecessor itself. Since one of $(0,0,...,0,1,1)$ is $(0,0,...,0,1)$, all of its other predecessors will be of the form $(z,m-z, z, ..., z, 1-z)$ for some $z \in \mathbb{Z}_m$. Suppose there exists nonzero $z$ such that this tuple has a predecessor. Then it must be true that $z+z+z+\cdots+z-1+z \equiv 0 \; \text{mod} \; m$, which would suggest $nz \equiv 1 \; \text{mod} \; m$. However, since $gcf(n,m)>1$, we cannot have this and we have a contradiction. Therefore, $L_m(n)=p^k$. 

  \textbf{(3)} We prove this via induction on $l$, where part (2) of this theorem serves as the basis case of $l=1$.
  
  \indent \textbf{Inductive step:} Assume that if $n=p^kn_1$, $m^*=p^{l-1}m_1^*$, $gcf(n_1, m_1^*)=1$, and $p \nmid n_1, m_1^*$ then $L_{m^*}(n) \leq p^{k-1}((l-1)(p-1)+1)=\gamma$. 
  
  \indent Now assume that $n=p^kn_1$ and $m=p^lm_1$ with $gcf(n_1, m_1)=1$ and $p \nmid n_1, m_1$. We set out to prove that $L_m(n) \leq \gamma+p^k-p^{k-1}$. Let $P_m(n)=d$. Without loss of generality, we may assume $p^k|d$. Then 
   \[a_{\gamma+d,s} \equiv a_{\gamma,s} \; \text{mod} \; \frac{m}{p}.\]
   Use this to write
   \[a_{\gamma+d,s} \equiv a_{\gamma,s}+\delta_s\frac{m}{p} \; \text{mod}\; m\]
   for some $0 \leq \delta_s <p$. Note also that 
   \[a_{\gamma+p^k-p^{k-1},s}=\sum_{i=1}^n a_{p^k-p^{k-1},i}a_{\gamma,s-i+1}.\]
   Since $a_{p^k-p^{k-1},i}$ is a binomial coefficient for every $i$, plug these in to see this sum is
   \begin{equation}\label{Eq_alt_breakdown}
   \sum_{i=1}^n\binom{p^k-p^{k-1}}{i-1}a_{\gamma,s-i+1}.
   \end{equation}
   Now consider $a_{\gamma+p^k-p^{k-1}+d,s}$:
   \[a_{\gamma+p^k-p^{k-1}+d,s}=\sum_{i=1}^n a_{p^k-p^{k-1},i}a_{\gamma+d,s-i+1}.\]
   Plugging in our values for the $a_{\gamma+d,s-i+1}$, this is equivalent to
   \[\sum_{i=1}^n a_{p^k-p^{k-1},i}(a_{\gamma,s-i+1}+\delta_{s-i+1}\frac{m}{p}) \; \text{mod} \; m.\]
   Separating the sums and plugging in $a_{p^k-p^{k-1},i}$, this is
   \[ \sum_{i=1}^n \binom{p^k-p^{k-1}}{i-1}a_{\gamma,s-i+1}  +\frac{m}{p}\sum_{i=1}^n \binom{p^k-p^{k-1}}{i-1}\delta_{s-i+1}\; \text{mod} \; m.\]
   Using Equation (\ref{Eq_alt_breakdown}), this is
   \[ a_{\gamma+p^k-p^{k-1},s}+\frac{m}{p}\sum_{i=1}^n \binom{p^k-p^{k-1}}{i-1}\delta_{s-i+1} \; \text{mod} \; m.\]
   Therefore, we only need that $\displaystyle{\sum_{i=1}^n \binom{p^k-p^{k-1}}{i-1}\delta_{s-i+1}} \equiv 0 \; \text{mod} \; p$ to see 
   \[L_m(n)\leq \gamma+p^k-p^{k-1}.\] 
   We consider two cases:
  
   \textbf{Case 1} $\mathbf{p}$ \textbf{is an odd prime:} Note that by Lemma \ref{binom_power_of_p_diff}, we have
   \[\sum_{i=1}^n \binom{p^k-p^{k-1}}{i-1}\delta_{s-i+1} \equiv \delta_s-\delta_{s-p^{k-1}}+\delta_{s-2p^{k-1}}-\cdots +\delta_{s-p^k+p^{k-1}} \; \text{mod} \; p\]
  when $p$ is an odd prime. Note that 
  \begin{equation}\label{Eq_coefficientsum_gamma}
  (a_{\gamma+d,s}-a_{\gamma,s})-(a_{\gamma+d,s-p^{k-1}}-a_{\gamma,s-p^{k-1}})+ \cdots + (a_{\gamma+d,s-p^k+p^{k-1}}-a_{\gamma,s-p^k+p^{k-1}}) 
  \end{equation}
  is equivalent to \[\frac{m}{p}(\delta_s-\delta_{s-p^{k-1}}+ \cdots +\delta_{s-p^k+p^{k-1}}) \; \text{mod} \; m.\]
 We can show $\delta_s-\delta_{s-p^{k-1}}+ \cdots +\delta_{s-p^k+p^{k-1}} \equiv 0 \; \text{mod} \; p$ by showing Equation (\ref{Eq_coefficientsum_gamma}) is congruent to $0 \; \text{mod} \; m$. Since we already know that this sum is equivalent to $0 \; \text{mod} \; \displaystyle{\frac{m}{p}}$, it suffices to show Equation (\ref{Eq_coefficientsum_gamma}) is congruent to $0 \; \text{mod} \; p$.
 Rewriting the sum in Equation (\ref{Eq_coefficientsum_gamma}) yields 
 \begin{equation}\label{Eq_2sums}
 \sum_{i=1}^{p-1}(-1)^ia_{\gamma+d,s-ip^{k-1}}-\sum_{i=1}^{p-1}(-1)^ia_{\gamma,s-ip^{k-1}}.
 \end{equation}
 Note that if $s \neq bp^{k-1}+1$ for some $b$, then this is equivalent to $0 \; \text{mod} \; p$ by Lemma \ref{p_dividing_multiplecoeff} and because $p^{k-1}|\gamma,d$. Assume then that $s=bp^{k-1}+1$ for some $b$. Then using $d=cp^k$, $\gamma+d=p^{k-1}((l-1)(p-1)+1+cp)$. So the first sum in Equation (\ref{Eq_2sums}) is
  \[\sum_{i=1}^{p-1}(-1)^ia_{\gamma+d,s-ip^{k-1}}= \sum_{i=1}^{p-1}a_{\gamma+d,(b-i)p^{k-1}+1},\]
  which by Lemma \ref{coefficient_to_smaller_n_broader}, is congruent to
 \begin{equation}\label{Eq_apply_period}
 \sum_{i=1}^{p-1}(-1)^ia^*_{(l-1)(p-1)+1+cp,b-i+1} \; \text{mod} \; p.
 \end{equation}
 By Proposition 4 of \cite{Breuer1}, $P_p(n)=p^kP_p(n_1)$. By Proposition 3.1 of \cite{Dular}, $P_p(n)|P_m(n)$. So
 \[p^kP_p(n_1)|cp^k\]
 and we conclude $P_p(n_1)|c$. Using this yields that (\ref{Eq_apply_period}) is 

 \[\sum_{i=1}^{p-1}(-1)^ia^*_{(l-1)(p-1)+1,b-i+1} \; \text{mod} \; p.\]
  Looking at our other sum in Equation (\ref{Eq_2sums}), 
 \[\sum_{i=1}^{p-1}(-1)^ia_{\gamma,s-ip^{k-1}}= \sum_{i=1}^{p-1}(-1)^ia_{\gamma,(b-i)p^{k-1}+1},\]
 which by Lemma \ref{coefficient_to_smaller_n_broader} is
 \[\sum_{i=1}^{p-1}(-1)^ia^*_{(l-1)(p-1)+1,b
 -i+1} \; \text{mod} \; p.\]
 So we have that Equation (\ref{Eq_2sums}) is
 \[\sum_{i=1}^{p-1}(-1)^ia_{\gamma+d,s-ip^{k-1}}-\sum_{i=1}^{p-1}(-1)^ia_{\gamma,s-ip^{k-1}} \equiv 0 \; \text{mod} \; p,\]
  which gives us that $L_m(n)\leq p^{k-1}(l(p-1)+1)$. 
 
  \textbf{Case 2} $\mathbf{p=2}$: We want to show $L_m(n) \leq 2^{k-1}(l+1)$. Write $d=2^kc$. Then the sum we are interested in is
   \[\sum_{i=1}^n \binom{p^k-p^{k-1}}{i-1}\delta_{s-i+1}=\sum_{i=1}^n \binom{2^{k-1}}{i-1} \delta_{s-i+1},\]
   which is equivalent to
   \[\delta_s+\delta_{s-2^{k-1}} \; \text{mod} \; 2\]
   Note that 
   \[(a_{\gamma+d,s}-a_{\gamma,s})+(a_{\gamma+d,s-2^{k-1}}-a_{\gamma,s-2^{k-1}}) \equiv \frac{m}{2}(\delta_s+\delta_{s-2^{k-1}} \; \text{mod}) \; m,\]
    so similar to before, $\delta_s+\delta_{s-2^{k-1}} \equiv 0 \; \text{mod} \; 2$ if 
    \[(a_{\gamma+d,s}+a_{\gamma+d,s-2^{k-1}})-(a_{\gamma,s}+a_{\gamma,s-2^{k-1}}) \equiv 0 \; \text{mod} \; 2.\]
     Notice that this is equivalent to $ 0 \; \text{mod} \; 2$ if $s \neq 2^{k-1}b+1$ for some $b$. Assume then that $s=2^{k-1}b+1$. Then 
   
   \[(a_{\gamma+d,s}+a_{\gamma+d,s-2^{k-1}})-(a_{\gamma,s}+a_{\gamma,s-2^{k-1}})\]\[= (a_{\gamma+d,2^{k-1}b+1}+a_{\gamma+d, 2^{k-1}(b-1)+1})-(a_{\gamma,2^{k-1}b+1}+a_{\gamma,2^{k-1}(b-1)+1})\]
   Using $\gamma+d=2^{k-1}((l-1)+1+2c)$, this is congruent to
   \[(a^*_{(l-1)+1+2c,b+1}+a^*_{(l-1)+1+2c,b})-(a^*_{(l-1)+1,b+1}+a^*_{(l-1)+1, b}) \; \text{mod} \; 2.\]
   Like before, we still have that $P_p(n_1)|c$, so this is 
   \[(a^*_{(l-1)+1,b+1}+a^*_{(l-1)+1, b})-( a^*_{(l-1)+1,b+1}+a^*_{(l-1)+1, b}) \; \text{mod} \; 2,\]
   which is equivalent to $0 \; \text{mod} \; 2$. Therefore, $L_m(n) \leq 2^{k-1}(l+1)$. 
 
\indent \textbf{(4)}: Assume $n=p_1^{k_1}p_2^{k_2} \cdots p_r^{k_t}n_1$ and $m=p_1^{l_1}p_2^{l_2} \cdots p_r^{l_t}m_1$ with $p_i \nmid n_1, m_1$ for every $i$ and $gcf(n_1,m_1)=1$. We wish to show that 
\[L_m(n)=max \{L_{p_i^{l_i}}(n) \; | \; 1 \leq i \leq t\}.\]
 Let $\gamma_i=L_{p_i^{l_i}}(n)$, $\gamma=max\{\gamma_i \; | \; 1 \leq i \leq t\}$, and $d=P_m(n)$. We will have that 
\[a_{\gamma_i+d,s} \equiv a_{\gamma_i,s} \; \text{mod} \; p_i^{l_i}\]
\[a_{1+d,s} \equiv a_{1,s} \; \text{mod} \; m_1\]
for every $s$ and every $i$. 
Then, since $\gamma \geq \gamma_i$ for every $i$, we have that 
\[a_{\gamma+d,s} \equiv a_{\gamma,s} \; \text{mod} \; p_i^{l_i}\]
\[a_{\gamma+d,s} \equiv a_{\gamma,s} \; \text{mod} \; m_1,\]
which gives $a_{\gamma+d,s} \equiv a_{\gamma,s} \; \text{mod} \; m$. 

\indent If we had that $L_m(n)< \gamma$, then $a_{\gamma+d-1,s} \equiv a_{\gamma-1,s} \; \text{mod} \; m$ would imply $a_{\gamma+d-1,s} \equiv a_{\gamma-1,s} \; \text{mod} \; p_i^{l_i}$ for every $i$. But $a_{\gamma+d-1, s} \not \equiv a_{\gamma-1,s} \; \text{mod} \; p_j^{l_j}$ for some $j$, which gives us a contradiction. Therefore, $L_m(n)=\gamma$.

\end{proof}

\end{document}